\documentclass[orivec]{llncs}

\usepackage[T1]{fontenc}
\usepackage{graphicx}
\usepackage{amssymb,amsmath}
\usepackage[multiple]{footmisc}
\usepackage{enumitem}

\providecommand{\bigjoin}{\mathop{%
\mathchoice{\raisebox{-2pt}{\huge$\Join$}}{\mbox{\LARGE$\Join$}}%
{\raisebox{0pt}{\Large$\Join$}}{\Join}}\displaylimits}%

\newcommand\restr[2]{{
  \left.\kern-\nulldelimiterspace 
  #1 
  \vphantom{\big|} 
  \right|_{#2} 
  }}

\DeclareMathOperator{\dom}{dom}
\DeclareMathOperator{\del}{del}
\DeclareMathOperator{\free}{free}
\newcommand{\dbar}{\mathord{\parallel}}

\title{Conjunctive Table Algebras}

\begin{document}

\author{Jens K\"{o}tters \and Stefan E. Schmidt}
\institute{Technische Universit\"{a}t Dresden, Dresden, Germany}
\maketitle

\begin{abstract}
Conjunctive table algebras are introduced and axiomatically characterized.
A conjunctive table algebra is a variant of SPJR algebra (a weaker form of relational algebra),
which corresponds to conjunctive queries with equality. The table operations relate to
logical operations (e.g. column deletion corresponds to existential quantification).
This enables a connection between database theory and algebraic logic, particularly cylindric algebras.
A comparison shows which cylindric algebra axioms hold in conjunctive table algebras,
which ones are modified, and which ones hold in addition.
\end{abstract}

\section{Introduction}
\label{intro0}
For the following discussion, we assume it is known what a first-order formula $\varphi$ is, and what it means for $\varphi$ to hold
in a structure $\mathfrak{A}$ under a variable assignment $\alpha$ (written as $\mathfrak{A} \models \varphi[\alpha]$).
A countably infinite set $\mathbf{var}$ contains the variables, and $x_1,x_2,x_3,\dots$ is a fixed enumeration of $\mathbf{var}$
(whereas $x$, $y$, $z$, $y_i$ or $z_i$ denote arbitrary variables). We only consider the relational setting,
where all atoms in $\varphi$ have the form $Rz_1\dots z_n$ or $x=y$, and $\mathfrak{A}$ is a \emph{relational structure},
i.e. $\mathfrak{A}=(A,(R^\mathfrak{A})_R)$, assigning a relation $R^{\mathfrak{A}}$ to each symbol $R$.

In this setting, two different concepts of solution set can be compared, which are founded in algebraic logic and database theory, respectively.
In each case, we consider $\mathfrak{A}$ to be fixed. First, we reproduce the relevant definitions for algebraic logic~\cite{Monk93}.
The "solution set" $\varphi^{\mathfrak{A}} := \{\alpha \in A^{\mathbf{var}} \mid \mathfrak{A} \models \varphi[\alpha]\}$
contains all variable assignments $\alpha:\mathbf{var} \rightarrow A$ which make $\varphi$ hold. The "solution sets"
are collected in the set $\mathit{Cs}\mathfrak{A}$, which is thus a subset of the power set $\mathfrak{P}(A^{\mathbf{var}})$.
Logical operations translate to set operations on $\mathit{Cs}\mathfrak{A}$. We have
$(\varphi \wedge \psi)^{\mathfrak{A}} = \varphi^{\mathfrak{A}} \cap \psi^{\mathfrak{A}}$,
$(\varphi \vee \psi)^{\mathfrak{A}} = \varphi^{\mathfrak{A}} \cup \psi^{\mathfrak{A}}$
and $(\neg \varphi)^{\mathfrak{A}} = -\varphi^{\mathfrak{A}}$ (complement in $A^{\mathbf{var}}$).
Special formulas $\mathbf{T}$ (tautology) and $\mathbf{F}$ (contradiction) satisfy $\mathbf{T}^{\mathfrak{A}}=A^{\mathbf{var}}$
and $\mathbf{F}^{\mathfrak{A}} = \emptyset$. Existential quantification over $x$ is described by a \emph{cylindrification operation} $C_x$,
i.e. $(\exists x \varphi)^{\mathfrak{A}} = C_x(\varphi^{\mathfrak{A}})$, and each equality atom $x\mathop{=}y$ is represented
by a \emph{diagonal} $D_{xy}=(x\mathop{=}y)^{\mathfrak{A}}$.
The algebra $\mathfrak{Cs}\mathfrak{A}:=(\mathit{Cs}\mathfrak{A},\cup,\cap,-,\emptyset,A^{\mathbf{var}},C_x,D_{xy})_{x,y \in \mathbf{var}}$
is a particular kind of cylindric set algebra, which is in turn a particular kind of cylindric algebra. Cylindric algebras are
defined by axioms (see e.g.~\cite{HMT71}).

The database-theoretic analog of $\varphi^{\mathfrak{A}}$ is
$\mathrm{res}_{\mathfrak{A}}(\varphi) := \{t \in A^{\mathrm{free}(\varphi)} \mid \mathfrak{A} \models \varphi[t]\}$,
where $\free(\varphi)$ is the set of free variables in $\varphi$, and $\mathfrak{A} \models \varphi[t]$ has the obvious meaning.
Here, we understand $\varphi$ as a relational calculus query, $\mathfrak{A}$ as a database, and $\mathrm{res}_{\mathfrak{A}}(\varphi)$
as a result table in the named perspective (cf.~\cite[Sect.\,3.2]{AHV95}), with entries in $A$ and the free variables as column names.
The set of column names is the \emph{table schema}, so a table with schema $X$ is an element of $\mathfrak{P}(A^X)$.
The set of tables with entries in $A$ is $\mathrm{Tab}(A) := \bigcup_{X \in \mathfrak{P}_{\mathrm{fin}}(\mathbf{var})} \mathfrak{P}(A^X)$,
where $\mathfrak{P}_{\mathrm{fin}}(\mathbf{var})$ contains the finite subsets of $\mathbf{var}$. All non-empty tables in $\mathrm{Tab}(A)$
thus have a finite schema, but there is only a single empty table; we assign it the infinite schema $\mathbf{var}$.
Naturally, $\mathrm{Tab}(A)$ is extended with the operations of Codd's relational algebra. One of these operations is the \emph{natural join},
given by $T_1 \Join T_2 := \{t \in A^{\mathrm{X_1} \cup \mathrm{X_2}} \mid t|_{X_1} \in T_1 \text{ and } t|_{X_2} \in T_2\}$ for nonempty
$T_1 \subseteq A^{X_1}$ and $T_2 \subseteq A^{X_2}$, as well as $\emptyset \Join T_2 = T_1 \Join \emptyset = \emptyset$. Further operations
are discussed in Sect.~\ref{ctas0}.

Imieli\'{n}ski and Lipski~\cite{IL84} have described a bridge between database theory and algebraic logic,
based on a function $h$, which identifies tables with sets in a cylindric set algebra.
We state it as $h:\mathrm{Tab}(A) \rightarrow \mathfrak{P}(A^{\mathbf{var}})$, defined by
$h(T):=\{\alpha \in A^{\mathbf{var}} \mid \alpha|_{\mathrm{schema}(T)} \in T\}$.
Notably, the natural join is represented by set intersection, i.e. $h(S \Join T) = h(S) \cap h(T)$.
Not every relational algebra operation matches a cylindric algebra operation,
but can be replicated on $\mathfrak{P}(A^{\mathbf{var}})$, and is preserved in this sense. Also, the "solution sets" are preserved,
i.e. $h(\mathrm{res}_{\mathfrak{A}}(\varphi)) = \varphi^{\mathfrak{A}}$.
Imieli\'{n}ski and Lipski demonstrate the practicality of this connection by applying results, that are available for cylindric algebras,
to relational algebra. However, one aspect deserves closer attention: For all tables $A^{\{x\}}$, $x \in \mathbf{var}$,
we have $h(A^{\{x\}}) = A^{\mathbf{var}}$, moreover $h(T) = h(T \Join A^{\{x\}})$ for all $T$ with $x \not\in \mathrm{schema}(T)$,
which shows that $h$ is not injective. What happens here is, that $h$ does not preserve the schema!
Using the set $\mathrm{NTup}(A) := \bigcup_{X \in \mathfrak{P}_{\mathrm{fin}}(\mathrm{var})} A^X$ of \emph{named tuples} over $A$,
we obtain a similar, but accurate set representation of $\mathrm{Tab}(A)$: the function $h^{*}:\mathrm{Tab}(A) \rightarrow \mathfrak{P}(\mathrm{NTup}(A))$,
with $h^{*}(T):=\{t \in \mathrm{NTup}(A) \mid t \text{ extends some } s \in T\}$, which contains all extensions of rows in $T$ (including the rows
themselves). However, while $h^{*}$ is injective, $\mathfrak{P}(\mathrm{NTup}(A))$ does not underlie a cylindric set algebra.

Imieli\'{n}ski and Lipski describe conditions~\cite[Thm.\,2]{IL84} under which $h$ acts as an embedding.
More precisely, $h$ is restricted to finite tables, and $A$ is required to be infinite. These restrictions are justified
under the conventional interpretation of $\mathfrak{A}$ as a database, where the relations $R^{\mathfrak{A}}$ represent
the tables in the database, and $A$ contains the possible table entries, like integers, strings or dates.
However, we want to support a different interpretation, more abstract technically, but more concrete conceptually,
where $\mathfrak{A}$ is a finite set, containing the actual objects described in a database, like customers, employees and orders.
In this case, $\mathfrak{A}$ could be obtained from a relational database by conceptual scaling~\cite{KE18}, a method used in
Formal Concept Analysis~\cite{GaW1999}. In this case, the relations $R^{\mathfrak{A}}$ describe either foreign key relations,
or virtual tables (i.e. views) which represent possible conditions in a \texttt{WHERE}-clause, like \texttt{Age >= 30} or
\texttt{Nationality='Spanish'} or \texttt{LastName LIKE 'S\%'}; the available relations depend on the scaling.
In this context, $\mathrm{Tab}(A)$ represents the space of result tables, whereas no assumption on the underlying data model are made.
Ultimately, as in formal logic, $\mathfrak{A}$ can be any relational structure, with infinite relations.

In this paper, we describe an alternative bridge between database theory and algebraic logic, which is not based on the function $h$.
While $h$ preserves relational algebra operations on $\mathfrak{P}(A^{\mathbf{var}})$, we take kind of an opposite approach,
and define analogs of cylindric algebra operations on $\mathrm{Tab}(A)$. However, we do not cover the supremum and complement
operations of cylindric algebras; so the resulting table algebra does not have the full expressivity of relational algebra;
it corresponds to SPJR algebra~\cite{AHV95}. On the other hand, we obtain an exact characterization
by axioms. The axioms take the schema into account, which was neglected by $h$, and are based on cylindric algebra axioms
as much as possible. All results of cylindric algebra, which are derived from the shared axioms, obviously hold in the table algebra;
also, the similarity between the two axiomatizations allows to use the theory of cylindric algebras as a blueprint for further development.

\section{Conjunctive Table Algebras}
\label{ctas0}
A \emph{table} over $G$ is a set $T \subseteq G^X$,
and $\mathrm{Tab}(G) = \bigcup \{\mathfrak{P}(G^X) \mid X \in \mathfrak{P}_{\mathrm{fin}(\mathbf{var})}\}$ is the set of tables over $G$.
The \emph{schema} of $T \in \mathfrak{P}(G^X)\setminus\{\emptyset\}$ is $X$, and each $t \in T$ is a \emph{row} of $T$, with \emph{entry} $t(x)$ in the
column with \emph{header} $x \in X$. The element $\emptyset \in \mathrm{Tab}(G)$ is the \emph{empty table}, and we assign it the schema $\mathbf{var}$.
The function $\mathrm{schema}:\mathrm{Tab}(G) \rightarrow \mathfrak{P}(\mathbf{var})$ maps each table to its schema.
We call $\mathrm{Tab}(G)[X] := G^X$ the \emph{$X$-slice} of $\mathrm{Tab}(G)$, whereas $\mathrm{Tab}^{*}(G)[X] := \mathrm{Tab}(G)[X]\setminus\{\emptyset\}$
contains precisely the tables with schema $X$.
The natural join of $T_1 \in \mathrm{Tab}(G)[X_1]$ and $T_2 \in \mathrm{Tab}(G)[X_2]$ (see Sect.~\ref{intro0})
is a table $T_1 \Join T_2 \in \mathrm{Tab}(G)[X_1 \cup X_2]$.

The natural join is associative, commutative, and idempotent (the latter means $T \Join T = T$), so by definition,
$(\mathrm{Tab}(G),\Join)$ is a semilattice. This means we have an implicit partial order on $\mathrm{Tab}(G)$ with the natural join as its infimum;
it is given by $T_1 \leq T_2 :\Leftrightarrow T_1 = T_1 \Join T_2$. The empty table $\emptyset$ is the least element in the table order, and the
table $\{\langle\rangle\} \in \mathrm{Tab}^{*}(G)[\emptyset]$, which contains only the \emph{empty tuple} $\langle\rangle \in G^{\emptyset}$,
is the greatest element. The following proposition characterizes the table order
(and explains why, from an order-perspective, the schema $\mathbf{var}$ suits the empty table).
\begin{proposition}
\label{torder0}
Let $T_1 \in \mathrm{Tab}(G)[X_1]$ and $T_2 \in \mathrm{Tab}(G)[X_2]$.
Then $T_1 \leq T_2 \Leftrightarrow (X_1 \supseteq X_2$
and $\{t|_{X_2} \mid t \in T_1\} \subseteq T_2)$.
\end{proposition}
\begin{proof}
If $T_1=\emptyset$, both sides of the equivalence hold. If $T_1 \neq \emptyset$ and $T_2 = \emptyset$, neither side holds.
For $T_1,T_2 \neq \emptyset$, the equivalence is straightforward to obtain from the definition of the natural join.
\qed
\end{proof}
Even with this characterization, the table order may seem artificial. The nature of the table order becomes clear, if each table $T$ is identified
with $h^{*}(T)=\{t \in \mathrm{NTup}(G) | s \leq t \text{ extends some } s \in T\}$, cf. Sect.~\ref{intro0}, as Prop.~\ref{tblord0} shows.
\begin{proposition}
\label{tblord0}
The function $h^{*}:\mathrm{Tab}(G) \rightarrow \mathfrak{P}(\mathrm{NTup}(G))$ defines an order embedding of $(\mathrm{Tab}(G),\leq)$
into $(\mathfrak{P}(\mathrm{NTup}(G)),\subseteq)$, i.e.
$T_1 \leq T_2 \Leftrightarrow \mathrm{tups}(T_1) \subseteq \mathrm{tups}(T_2)$ for all $T_1,T_2 \in \mathrm{Tab}(G)$.
\end{proposition}
\begin{proof}
Let $T_1 \in \mathrm{Tab}(G)[X_1]$ and $T_2 \in \mathrm{Tab}(G)[X_2]$. "$\Rightarrow$": Assume $T_1 \leq T_2$. If $t \in \mathrm{tups}(T_1)$,
then $t|_{X_1} \in T_1$, and thus $t|_{X_2} = (t|_{X_1})|_{X_2} \in T_2$ by assumption, which means $t \in \mathrm{tups}(T_2)$.
This shows $\mathrm{tups}(T_1) \subseteq \mathrm{tups}(T_2)$.
"$\Leftarrow$": Assume $\mathrm{tups}(T_1) \subseteq \mathrm{tups}(T_2)$. If $t \in T_1$, then also
$t \in \mathrm{tups}(T_1) \subseteq \mathrm{tups}(T_2)$, so $X_2 \subseteq X_1$ and $t|_{X_2} \in T_2$.
This shows $T_1 \leq T_2$.
\qed
\end{proof}
In the spirit of cylindric set algebras, we translate logical operations to table operations (cf. Sect~\ref{intro0}).
We restrict ourselves to conjunctive calculus queries with equality; they are represented by \emph{primitive positive formulas}
(i.e. formulas built from atoms using $\wedge$ and $\exists$). Each equality atom $x\mathop{=}y$ is represented by the
\emph{equality table} $E_{xy}:=\{t \in G^{\{x,y\}} \mid t(x)=t(y)\}$, because $\mathrm{res}_{\mathfrak{A}}(x\mathop{=}y) = E_{xy}$.
For $x=y$, this table has only a single column. Conjunction is described by the natural join, because
$\mathrm{res}_{\mathfrak{A}}(\varphi \wedge \psi) = \mathrm{res}_{\mathfrak{A}}(\varphi) \Join \mathrm{res}_{\mathfrak{A}}(\psi)$.
Existential quantification over $x$ is described by a \emph{delete operation},
defined by $\mathrm{del}_x(T) := \{t|_{X\setminus\{x\}} \mid t \in T\}$ for $T \in \mathrm{Tab}(G)[X]$,
because $\mathrm{res}_{\mathfrak{A}}(\exists x \varphi) = \mathrm{del}_x(\mathrm{res}_{\mathfrak{A}}(\varphi))$.
It deletes the $x$-column of $T$, if present, and otherwise leaves $T$ unchanged.
The table algebra $\mathbf{Tab}(G):=(\mathrm{Tab}(G),\Join,\emptyset,\{\langle\rangle\},\mathrm{del}_x,E_{xy},\mathrm{schema})_{x,y \in \mathbf{var}}$
is an extension of the bounded semilattice $(\mathrm{Tab}(G),\Join,\emptyset,\{\langle\rangle\})$.
\begin{definition}
A \emph{conjunctive table algebra (with equality)} with base $G$ is a subalgebra of $\mathbf{Tab}(G)$.
\end{definition}
We now consider some derived operations, which exist in every conjunctive table algebra (with equality).
In analogy to generalized cylindrifications and generalized diagonals, which are defined in cylindric algebras (see~\cite{HMT71}),
we define the \emph{generalized deletion}
$\mathrm{del}_{X}(T) := \mathrm{del}_{z_1}\dots\mathrm{del}_{z_n}(T)$ for all $X:=\{z_1,\dots,z_n\} \in \mathfrak{P}_{\mathrm{fin}}(\mathbf{var})$,
and the \emph{generalized equality table} $E_{\varrho} := \bigjoin_{(x,y) \in \theta} E_{xy}$ for all finite
$\varrho \subseteq \mathbf{var} \times \mathbf{var}$. In particular, we have $\mathrm{del}_{\emptyset}(T)=T$ and $E_{\emptyset}=\{\langle\rangle\}$.
The \emph{projection operation} is defined by $\mathrm{proj}_Y(T) := \mathrm{del}_{\mathrm{X\setminus Y}}(T)$ for all $T \in \mathrm{Tab}(G)[X]$
and $Y \subseteq X$. The \emph{duplication operation} is defined by $\mathrm{dup}_{xy}(T) := T \Join E_{xy}$ for all $T \in \mathrm{Tab}(G)[X]$,
$x \in X$ and $y \in \mathbf{var}\setminus X$. It creates a new column $y$, which is a copy of $x$. If the column $x$ is then deleted,
we have effectively renamed the column $x$ into $y$. The \emph{renaming operation} is thus defined by
$\mathrm{rnm}_{xy}(T) := \mathrm{del}_x(T \Join E_{xy})$ for all $T \in \mathrm{Tab}(G)[X]$, $x \in X$ and $y \in \mathbf{var}\setminus X$.
By repeated application, an arbitrary one-to-one renaming of columns can be performed. We call a table algebra a \emph{DPJR algebra}
if it is closed under duplication, projection, natural join, and renaming. So every conjunctive table algebra (with equality)
is a DPJR algebra. Note that $\mathrm{del}_x$ can be obtained from projection; so conjunctive table algebras are characterized as
DPJR algebras which contain $\emptyset$, $\{\langle\rangle\}$ and $D_{xy}$ for $x,y \in \mathbf{var}$. We can improve a bit on the characterization,
and only require DPJR algebras which contain $\emptyset$ and $D_{xx}$ for some $x \in \mathbf{var}$: the tables $D_{yz}$ can be obtained from $D_{xx}$
by renaming and duplication, and we have $\{\langle\rangle\} = \mathrm{proj}_{\emptyset}(D_{xx})$.

A comparison with the well-established SPJR algebra~\cite{AHV95} is in order. SPJR algebra defines two kinds of \emph{selection operations},
in addition to projection, natural join and renaming. The first kind of selection is defined by $\sigma_{x=y}(T) := \{t \in T \mid t(x)=t(y)\}$
for all $T \in \mathrm{Tab}(G)[X]$ and $x,y \in X$. It can be derived in a conjunctive table algebra (with equality),
since $\sigma_{x=y}(T) = T \Join E_{xy}$. The definition is the same as for duplication above, except that we require $x,y \in X$ here.
DPJR algebra is equivalent to SPJR algebra with only this first kind of selection, since $\sigma_{x=y}$ and $\mathrm{dup}_{xy}$ can be
defined in terms of each other: we have $\sigma_{x=y}(T) = T \Join (\mathrm{dup}_{xy}(\mathrm{proj}_{\{x\}}(T)))$ and
$\mathrm{dup}_{xy}(T) = \sigma_{x=y}(T \Join \mathrm{rnm}_{xy}(\mathrm{proj}_{\{x\}}(T)))$. The second kind of selection is defined by
$\sigma_{x=g}(T) := \{t \in T \mid t(x)=g\}$ for all $T \in \mathrm{Tab}(G)[X]$, $x \in X$ and $g \in G$. Conjunctive table algebras are
generally not closed under this second kind of selection (as an example, consider the smallest subalgebra of $\mathbf{Tab}(G)$,
which contains only the generalized equality tables $E_{\varrho}$ and the empty table).

We can represent projection, renaming and duplication by a single operation. The \emph{outer composition} is defined by
$T \circ \lambda := \{t \circ \lambda \mid t \in T\}$ for all $X,Y \in \mathfrak{P}_{\mathrm{fin}}(\mathbf{var})$, $T \in \mathrm{Tab}(G)[Y]$
and $\lambda:X \rightarrow Y$. We then have $T \circ \lambda \in \mathrm{Tab}(G)[X]$.
The projection $\mathrm{proj}_X$ is represented by the \emph{natural inclusion} $\iota_X:X \rightarrow Y$
(defined by $\iota_X(x)=x$ for all $x \in X$, where $X \subseteq Y$ is assumed); we have $\mathrm{proj}_X(T) = T \circ \iota_X$
for all $T \in \mathrm{Tab}(G)[Y]$.
A bijection $\xi:X \rightarrow Y$ renames column $\xi(x)$ to $x$ for each $x \in X$, i.e. $\xi$ maps the new column names to the old column names.
This generalizes the single-column renaming defined above: we obtain $\mathrm{rnm}_{yx}(T) = T \circ \xi$ for the bijection
$\xi:X \rightarrow Y$ with $X=(Y\setminus\{y\})\cup\{x\}$, $\xi(x)=y$ and $\xi(z)=z$ for $z \in Y\setminus\{y\}$.
Finally, we call $\delta:X \rightarrow Y$ a \emph{folding} if $Y \subseteq X$ and $\delta(x)=x$ for all $x \in Y$. Then $T \circ \delta$
has all columns of $T$, and additionally an $x$-column for all $x \in X\setminus Y$, which is a copy of the $\delta(x)$-column.
This generalizes the single-column duplication defined above. For composition with arbitrary $\lambda:X \rightarrow Y$, we consult the following lemma.
\begin{lemma}[Decomposition Lemma]
Every function $\lambda:X \rightarrow Y$ has a decomposition $\lambda = \iota \circ \xi \circ \delta$ into a folding $\delta:X \rightarrow Z_1$,
a bijection $\xi:Z_1 \rightarrow Z_2$, and a natural inclusion $\iota:Z_2 \rightarrow Y$.
\end{lemma}
\begin{proof}
Let $Z_1 \subseteq X$ be minimal with $\lambda(Z_1)=\lambda(X)=:Z_2$. Then for each $x \in X$ there is a unique $\delta(x) \in Z_1$ with
$\lambda(x)=\lambda(\delta(x))$. We thus obtain a folding $\delta:X \rightarrow Z_1$ with $\lambda = \lambda|_{Z_1} \circ \delta$.
Moreover, the function $\lambda_{Z_1}^{Z_2}:Z_1 \rightarrow Z_2$, given by $\lambda|_{Z_1}^{Z_2}(z)=\lambda(z)$ for $z \in Z_1$,
is a bijection, and we have $\lambda|_{Z_1} = \iota_{Z_2} \circ \lambda|_{Z_1}^{Z_2}$, where $\iota_{Z_2}:Z_2 \rightarrow Y$ is the natural inclusion.
\qed
\end{proof}
We also note that composition can be sequentialized, i.e. $T \circ (\nu \circ \mu) = (T \circ \nu) \circ \mu$ for all
$T \in \mathrm{Tab}(G)[Z]$, $\nu:Y \rightarrow Z$ and $\mu:X \rightarrow Y$. Applied to the decomposition lemma, this means
$T \circ \lambda = T \circ (\iota \circ \xi \circ \delta) = ((T \circ \iota) \circ \xi) \circ \delta$; or put in words, composition with $\lambda$
amounts to performing a projection, then a renaming, and then a duplication. This shows that every DPJR algebra, and thus every
conjunctive table algebra (with equality), is closed under composition.

\section{Projectional Semilattices}
\label{pslat0}
In this section, we formulate axioms for conjunctive table algebras (with equality), in the style of the axioms for cylindric algebras.
\begin{definition}
A \emph{projectional semilattice} is an algebra $(V,\wedge,0,1,c_x,d_{xy},\dom)_{x,y \in \mathbf{var}}$
consisting of an infimum operation $\wedge$, a \emph{bottom element} $0$, a \emph{top element} $1$,
a \emph{cylindrification} $c_x:V \rightarrow V$ for each $x \in \mathbf{var}$,
a \emph{diagonal} $d_{xy} \in V$ for each $(x,y) \in \mathbf{var} \times \mathbf{var}$,
and a \emph{domain function} $\dom:V \rightarrow \mathfrak{P}(\mathbf{var})$,
such that the axioms
\begin{enumerate}[label={\bf(PS\arabic*)},start=0,leftmargin=1.5cm]
\item \label{psaxslat0} $(V,\wedge,0,1)$ is a bounded semilattice
\item \label{psaxzero0} $c_x(0) = 0$
\item \label{psaxproj0} $u \leq c_x(u)$
\item \label{psaxdist0} $c_x(u \wedge c_x(v)) = c_x(u) \wedge c_x(v)$
\item \label{psaxcomm0} $c_x(c_y(u)) = c_y(c_x(u))$
\item \label{psaxdim0} $u \neq 0 \Rightarrow (u \neq c_x(u) \Leftrightarrow u \leq d_{xx})$
\item \label{psaxdxyz0} $x \neq y,z \Rightarrow d_{yz} = c_x(d_{yx} \wedge d_{xz})$
\item \label{psaxequal0} $x \neq y \Rightarrow d_{xy} \wedge c_x(d_{xy} \wedge u) \leq u$
\item \label{psaxfinite0} $u \neq 0 \Rightarrow \dom(u) \text{ finite}$
\item \label{psaxdom0} $\dom(u) = \{x \in \mathbf{var} \mid u \leq d_{xx}\}$
\item \label{psaxone0} $\dom(u) = \emptyset \Rightarrow u = 1$
\item \label{psaxdxx0} $d_{xx} \neq 0$
\item \label{psaxsym0} $d_{xy} = d_{yx}$
\end{enumerate}
hold for all $u,v \in V$ and $x,y,z \in \mathbf{var}$.
\end{definition}

\begin{theorem}
Every conjunctive table algebra (with equality) is a projectional semilattice.
\end{theorem}
\begin{proof}
We have obtained~\ref{psaxslat0} in Sect.~\ref{ctas0}. Axioms~\ref{psaxzero0} and~\ref{psaxproj0}, which state $\mathrm{del}_x(\emptyset)=\emptyset$ 
and $T \leq \mathrm{del}_x(T)$, follow from the definitions of $\del_x$ and $\leq$ in Sect.~\ref{ctas0}.
If $x \not\in \mathrm{schema}(T)$, we have $\del_x(S \Join T) = \del_x(S) \Join T$,
i.e. the result of the join is not affected by whether we delete $x$ before
or afterwards. Specifically, setting $T:=\mathrm{del}_x(R)$ for any table $R$, we obtain
$\del_x(S \Join \del_x(R)) = \del_x(S) \Join \del_x(R)$, which shows~\ref{psaxdist0}.
Axiom~\ref{psaxcomm0} is easy to see, and to verify~\ref{psaxdim0} and~\ref{psaxdom0}, we use
$T \leq E_{xx} \Leftrightarrow x \in \mathrm{schema}(T)$. We finally show~\ref{psaxequal0},
the remaining axioms are easy to see. The table $S := E_{xy} \Join T$ has identical columns $x$ and $y$.
So deleting $x$ and then duplicating $y$ into $x$ (by joining with $E_{xy}$) restores the table:
$E_{xy} \Join \del_x(S) = S$. Thus $E_{xy} \Join \del_x(E_{xy} \Join T) = E_{xy} \Join T \leq T$.
\qed
\end{proof}
The axioms~\ref{psaxzero0}, \ref{psaxproj0}, \ref{psaxdist0}, \ref{psaxcomm0} and \ref{psaxdxyz0}
are precisely the axioms $(C_1)$, $(C_2)$, $(C_3)$, $(C_4)$ and $(C_6)$ for cylindric algebras,
as printed e.g. in~\cite{HMT71}. Axiom~\ref{psaxslat0} corresponds to axiom $(C_0)$, which states that
cylindric algebras are (extensions of) Boolean algebras; obviously, axiom $(C_0)$ itself does not apply here, because
projectional semilattices do not provide supremum or complement operations. 
Likewise, axiom $(C_7)$ in~\cite{HMT71} involves the complement operation,
so it does not apply here; but interestingly, axiom \ref{psaxequal0} was used as an earlier version of $(C_7)$,
as related by Henkin et al. (cf. the footnote on p.\,162f. in~\cite{HMT71}). Solely the axiom $(C_5)$,
which states $d_{xx}=1$, contrasts strongly with the projectional semilattice axioms; the axiom~\ref{psaxdim0}
specifies the exact conditions for when $u \leq d_{xx}$ holds; we shall get back to~\ref{psaxdim0} shortly.
Axiom~\ref{psaxdom0} defines the function $\dom:V \rightarrow \mathfrak{P}(\mathbf{var})$, which abstractly
captures the table schema: for a table $T$, we have $x \in \mathrm{schema}(T)$ if and only if $T \leq E_{xx}$.
One of the fundamental notions for cylindric algebras is the \emph{dimension set}~\cite{HMT71};
we state it here as $\dim(u):=\{x \in \mathbf{var} \mid c_x(u) \neq u\}$ (although conventionally the notation $\Delta(u)$ is used).
Axiom~\ref{psaxdim0} thus states that $\dim(u)=\dom(u)$ for all $u \neq 0$ (we have $\dim(0)=\emptyset$ by~\ref{psaxzero0},
and $\dom(0)=\mathbf{var}$ by~\ref{psaxdom0}). There was originally an axiom $(C_8)$ for cylindric algebras,
which stated that all elements have a finite dimension set~\cite[p.\,10]{HMT71};
with the previous observation, we can recognize~\ref{psaxfinite0} as the axiom $(C_8)$.
Likewise, axiom~\ref{psaxone0} states that $\dim(u)=\emptyset$ only holds for $0$ and $1$;
the statement does not hold for cylindric algebras in general, but it holds for all regular cylindric set algebras
(see~\cite[Sect.\,12]{Monk93} for a definition of regularity); and it may be noted that Imieli\'{n}ski and Lipski's "embedding" $h$
(cf. Sect.~\ref{ctas0}) identifies every table algebra with a regular cylindric set algebra. Axiom~\ref{psaxdxx0} excludes the
degenerate table algebra $\mathbf{Tab}(\emptyset)$, which contains only $\emptyset$ and $\{\langle\rangle\}$.
While~\ref{psaxsym0} is derived for cylindric algebras~\cite[Thm.\,1.3.1]{HMT71}, it can not be derived from~\ref{psaxslat0} to~\ref{psaxdxx0};
to see this, consider the table algebra $\mathbf{Tab}(\{g\})$ over a singleton set,
with the alternative "bogus diagonal" $d_{xy}:=E_{xx}$; it satisfies~\ref{psaxslat0} to~\ref{psaxdxx0}, but not~\ref{psaxsym0}.

We now translate a few notions, that we have defined for tables in Sect.~\ref{ctas0}, into the abstract setting.
For each $X \in \mathfrak{P}_{\mathrm{fin}}(\mathbf{var})$,
the set $\mathbf{V}^{*}[X] := \{v \in V \mid \dom(v)=X\}$ collects all elements with domain $X$,
and $\mathbf{V}[X] := \mathbf{V}^{*} \cup \{0\}$ is the \emph{$X$-slice} of $\mathbf{V}$.
By~\ref{psaxcomm0}, the \emph{generalized cylindrification} $C_{\{z_1,\dots,z_n\}}(u) := c_{z_1}\dots c_{z_n}(u)$ is well-defined
for all $u \in V$. In particular, $C_{\emptyset}(u)=u$. The \emph{generalized diagonal} is defined by
$e_{\varrho} := \bigwedge_{(x,y) \in \varrho} d_{xy}$ for all finite $\varrho \subseteq \mathbf{var} \times \mathbf{var}$.
In particular, $e_{\emptyset} = 1$ is the empty infimum, and we have $e_{\lambda} = \bigwedge_{x \in X} d_{x\lambda(x)}$ for
a function $\lambda:X \rightarrow Y$, treating $\lambda$ as a relation. Defining \emph{outer composition} on $\mathbf{V}$ is
a bit technical. We call $\lambda:X \rightarrow Y$ \emph{domain-disjoint}, if $X \cap Y = \emptyset$. Concretely, for $T \in \mathrm{Tab}(G)[Y]$
and domain-disjoint $\lambda:X \rightarrow Y$, we have $T \circ \lambda = \del_X(T \Join E_{\lambda})$;
and for arbitrary $\lambda = \nu \circ \mu$, we have $T \circ \lambda = T \circ \nu \circ \mu$. This motivates the two-stage definition
\begin{align}
\label{outercomp0}
u \odot \lambda := \left\{\begin{array}{cl}
C_Y(u \wedge e_{\lambda}) &\quad \text{if $\lambda$ is domain-disjoint} \\
u \odot \xi_{XY} \odot (\xi_{XY}^{-1} \circ \lambda) &\quad \text{otherwise}
\end{array}\right. \quad,
\end{align}
where $\xi_{XY}:Z_{XY} \rightarrow Y$ is some fixed bijection, depending only on $X$ and $Y$,
and defined on a set $Z_{XY} \in \mathfrak{P}_{\mathrm{fin}}(\mathbf{var})$ with $Z_{XY} \cap (X \cup Y) = \emptyset$.
Then $\lambda = \xi_{XY} \circ \xi_{XY}^{-1}$ is a decomposition into domain-disjoint functions,
and~\eqref{outercomp0} is thus well-defined. We assume that~\ref{outercomp0} defines different operations $\odot_{XY}$,
one for each pair $(X,Y) \in \mathfrak{P}_{\mathrm{fin}}(\mathbf{var}) \times \mathfrak{P}_{\mathrm{fin}}(\mathbf{var})$,
but denote them by the same symbol $\odot$ for better readability. Proposition~\ref{ocomp0} will show that, as would be expected,
the definition in~\eqref{outercomp0} does not depend on the particular choice of the $\xi_{XY}$,
and uniquely characterizes the family $(\odot_{XY})_{X,Y \in \mathfrak{P}_{\mathrm{fin}}(\mathbf{var})}$ by its properties,
which makes the above definition obsolete.

We present a series of propositions, concerning the basic operations (Prop.~\ref{derived0}), domains/slices (Prop.~\ref{domains0}),
generalized cylindrification (Prop.~\ref{gency0}), genalized diagonals (Prop.~\ref{eqcomp0}), and the latter two together (Prop.~\ref{gencydiag0}),
before we can characterize outer composition (Prop.~\ref{ocomp0}). The final result, which is obvious for tables, can only be shown at the
end of this section, and fully integrates outer composition with the remaining operations (Prop.~\ref{proj0}).
The proofs can be found in Sect.~\ref{proofs0} in the appendix.
\begin{proposition}
\label{derived0}
\begin{enumerate}[label=\roman*),leftmargin=1cm]
\item \label{derived1} $c_x(c_x(v)) = c_x(v)$
\item \label{derived2} $u \leq v \Rightarrow c_x(u) \leq c_x(v)$
\item \label{derived3} if $x \not\in \dom(v)$, then $c_x(u \wedge v) = c_x(u) \wedge v$
\item \label{derived4} if $x \neq y$ and $u \leq d_{xy}$, then $d_{xy} \wedge c_x(u) = u$
\item \label{derived5} if $x \neq y$, then $c_x(d_{xy}) = d_{yy}$
\item \label{derived6} $d_{xz} \wedge d_{zy} \leq d_{xy}$
\end{enumerate}
\end{proposition}
\begin{proposition}
\label{domains0}
\begin{enumerate}[label=\roman*),leftmargin=1cm]
\item \label{domains1} $d_{xy} \in \mathbf{V}^{*}[\{x,y\}]$
\item \label{domains2} if $u \in \mathbf{V}[X]$ and $v \in \mathbf{V}[Y]$, then $u \wedge v \in \mathbf{V}[X \cup Y]$
\item \label{domains3} if $u \in \mathbf{V}^{*}[Y]$, then $C_Z(u) \in \mathbf{V}^{*}[Y \setminus Z]$
\item \label{domains4} if $\varrho \subseteq X \times Y$, then $e_{\varrho} \in \mathbf{V}^{*}[\mathrm{field}(\varrho)]$,\\
where $\mathrm{field}(\varrho):=\bigcup \{\{x,y\} \mid (x,y)\in\varrho\}$
\end{enumerate}
\end{proposition}
\begin{proposition}
\label{gency0}
\begin{enumerate}[label=\roman*),leftmargin=1cm]
\item \label{gency1} $C_X(0) = 0$
\item \label{gency2} $C_X(C_Y(u)) = C_Y(C_X(u))$
\item \label{gency3} if $Z \cap \dom(v) = \emptyset$, then $C_Z(u \wedge v) = C_Z(u) \wedge v$
\end{enumerate}
\end{proposition}
\begin{proof}
Statements~\ref{gency1}, \ref{gency2} and~\ref{gency3} are easily obtained from~\ref{psaxzero0}, \ref{psaxcomm0}
and Prop.~\ref{derived0}\ref{derived3}, respectively.
\qed
\end{proof}
\begin{proposition}
\label{eqcomp0}
Let $\mu:X \rightarrow Z$ and $\nu: Z \rightarrow Y$.
\begin{enumerate}[label=\roman*),leftmargin=1cm]
\item \label{eqcomp1} $e_{\mu} \wedge e_{\nu} \leq e_{\nu \circ \mu}$
\item \label{eqcomp2} $e_{\mu} \wedge e_{\nu} = e_{\nu\circ\mu} \wedge e_{\nu}$
\item \label{eqcomp3} if $\mu$ is a folding, then $e_{\mu} \wedge e_{\nu} = e_{\nu\circ\mu}$
\end{enumerate}
\end{proposition}
\begin{proposition}
\label{gencydiag0}
\begin{enumerate}[label=\roman*),leftmargin=1cm]
\item \label{gencydiag1} $C_X(u \wedge e_{\lambda}) = u$ for all domain-disjoint $\lambda:X \rightarrow Y$ and $u \in \mathbf{V}^{*}[Y]$
\item \label{gencydiag2} $C_X(u) \wedge e_{\lambda} = u$ for all domain-disjoint $\lambda:X \rightarrow Y$ and $u \leq e_{\lambda}$
\end{enumerate}
\end{proposition}
\begin{proposition}
\label{ocomp0}
Let $\mathbf{V}$ be a projectional semilattice. The family $(\odot_{XY})_{X,Y \in \mathfrak{P}_{\mathrm{fin}}(\mathbf{var})}$ of outer compositions
is the unique family of operations $\odot_{XY}:\mathbf{V}^{*}[Y] \times Y^X \rightarrow \mathbf{V}^{*}[X]$ such that
\begin{enumerate}[label={\bf(PSE\arabic*)},leftmargin=1.5cm]
\item \label{pseaxdis0} $u \odot_{XY} \lambda = C_{Y}(u \wedge e_{\lambda})$ if $\lambda$ is domain-disjoint
\item \label{pseaxsemi0} $u \odot_{XY} (\nu \circ \mu) = (u \odot_{ZY} \nu) \odot_{XZ} \mu$
\end{enumerate}
hold for all $u \in \mathbf{V}[Y]$, $\lambda:X \rightarrow Y$, $\mu:X \rightarrow Z$ and $\nu:Z \rightarrow Y$.
\end{proposition}
\begin{proposition}
\label{proj0}
Let $u \in \mathbf{V}^{*}[Y]$.
\begin{enumerate}[label=\roman*),leftmargin=1cm]
\item \label{proj1} $u \odot \sigma \odot \sigma^{-1} = u$ for all domain-disjoint bijections $\sigma:X \rightarrow Y$
\item \label{proj2} $u \odot \iota_X = C_{Y\setminus X}(u)$ for all inclusions $\iota_X:X \rightarrow Y$
\item \label{proj3} $u \odot \delta = u \wedge e_{\delta}$ for all foldings $\delta:X \rightarrow Y$
\end{enumerate}
\end{proposition}
\section{Representation Theorem}
\label{rep0}
In order to use a result from another paper~\cite{KS22a}, we have to introduce a variant of the outer composition.
We consider a \emph{finite partial transformation} of $\mathbf{var}$ to be a finite relation $\lambda \subseteq \mathbf{var} \times \mathbf{var}$
that is \emph{functional}, i.e. if $(x,y) \in \lambda$ and $(x,y') \in \lambda$ then $y = y'$. The set of all finite partial transformations
of $\mathbf{var}$ is denoted by $\mathcal{T}_{\mathrm{fp}}(\mathbf{var})$. The pair $(\mathcal{T}_{\mathrm{fp}}(\mathbf{var}),\circ)$ is
a semigroup with the usual composition of relations. For given $\lambda \in \mathcal{T}_{\mathrm{fp}(\mathbf{var})}$
and $Y \in \mathfrak{P}_{\mathrm{fin}}(\mathbf{var})$, the function $\lambda\dbar^Y:\lambda^{-1}(Y) \rightarrow Y$ is given
by $\lambda\dbar^Y(x):=\lambda(x)$. The equation $(\lambda \circ \mu)\dbar^Y = \lambda\dbar^Y \circ \mu\dbar^{\lambda^{-1}(Y)}$
relates composition on $\mathcal{T}_{\mathrm{fp}}(\mathbf{var})$ with function composition.

We now combine the family $(\odot_{XY})_{X,Y \in \mathfrak{P}_{\mathrm{fin}}(\mathbf{var})}$
into a single, total operation $\cdot: V \times \mathcal{T}_{\mathrm{fp}}(\mathbf{var}) \rightarrow V$, defined by
$u \cdot \lambda := u \odot \lambda\dbar^Y$ for all $u \in \mathbf{V}^{*}[Y]$, and furthermore $0 \cdot \lambda := 0$.
If $u \in \mathbf{V}^{*}[Y]$, then $u \cdot \lambda \in \mathbf{V}^{*}[\lambda^{-1}(Y)]$ by Prop.~\ref{ocomp0}.
Also, we obtain $u \cdot (\lambda \circ \mu) = u \odot (\lambda \circ \mu)\dbar^{Y}
= u \odot (\lambda\dbar^{Y} \circ \mu\dbar^{\lambda^{-1}(Y)}) = u \odot \lambda\dbar^{Y} \odot \mu\dbar^{\lambda^{-1}(Y)} = u \cdot \lambda \cdot \mu$
using the above equation and~\ref{pseaxsemi0}. Algebraically, the latter property means that
$\cdot:V \times \mathcal{T}_{\mathrm{fp}}(\mathbf{var}) \rightarrow V$ is a \emph{semigroup action}.
The relations $\pi_X \in \mathcal{T}_{\mathrm{fp}}(\mathbf{var})$, given by $\pi_X:=\{(x,x) \mid x \in X\}$,
are called \emph{local identities}; for all $u \in \mathbf{V}^{*}[Y]$, we obtain $u \cdot \pi_Y = u \odot \iota_Y = u$
from Prop.~\ref{proj0}\ref{proj2}. The semigroup $(\mathcal{T}_{\mathrm{fp}}(\mathbf{var}),\circ)$ naturally extends to a monoid
$(\mathcal{T}_{\mathrm{fp}}(\mathbf{var}) \cup \{\pi_{\mathbf{var}}\},\circ,\pi_{\mathbf{var}})$,
i.e. $\pi_{\mathbf{var}}$ is the \emph{global identity}, and likewise we have $u \cdot \pi_{\mathbf{var}} = u \odot \iota_Y = u$
for $u \in \mathbf{V}^{*}[Y]$; this means that the semigroup action naturally extends to a \emph{monoid action}
$\cdot:V \times (\mathcal{T}_{\mathrm{fp}}(\mathbf{var}) \cup \{\pi_{\mathbf{var}}\}) \rightarrow V$.
\begin{definition}
Let $\mathbf{V}=(V,\wedge,0,1,c_x,d_{xy},\mathrm{dom})_{x,y \in \mathbf{var}}$ be a projectional semilattice.
The \emph{orbital extension} of $\mathbf{V}$ is the algebra $(V,\wedge,0,1,\cdot,c_x,d_{xy},\mathrm{dom})_{x,y \in \mathbf{var}}$.
\end{definition}
Orbital semilattices have been introduced in~\cite{KS22a}. The main result of that paper states that orbital semilattices
are isomorphic to orbital table algebras (which use the semigroup action instead of deletion). Using the same isomorphism,
we can show that projectional semilattices are isomorphic to conjunctive table algebras with equality (Thm.~\ref{iso0}).
All that remains to do is verify the orbital semilattice axioms for the orbital extension. These axioms are printed below,
as they appear in~\cite{KS22a}; w.r.t. axiom~\ref{axiomdxy0}, we note that $\tfrac{xx}{xy}$ is a shorthand
for the partial finite transformation $\delta = \{(x,x),(y,x)\}$ (i.e. $\delta(x)=\delta(y)=x$).
\begin{theorem}
\label{orbital0}
The orbital extension of a projectional semilattice $\mathbf{V}$ satisfies the orbital semilattice axioms, i.e.
\begin{enumerate}[label=\bf(A\arabic*),leftmargin=1.5cm]
\item \label{axiomone0} $u \neq 0 \Rightarrow u \cdot \pi_{\emptyset} = 1$
\item \label{axiomzero0} $0 \cdot \lambda = 0$
\item \label{axiomdist0} $\dom(u) \subseteq Z \Rightarrow (u \wedge v) \cdot \pi_Z = u \wedge (v \cdot \pi_Z)$
\item \label{axiomproj0} $u \leq u \cdot \pi_Z$
\item \label{axiommult0} $u \leq v \Rightarrow u \cdot \lambda \leq v \cdot \lambda$
\item \label{axiomequal0} $(u \leq d_{xy} \text{ and } u \neq 0 \text{ and } x \neq y) \Rightarrow u = (u \cdot \pi_{\dom(u)\setminus\{y\}}) \wedge d_{xy}$
\item \label{axiomsemi0} $u \cdot \lambda \cdot \mu = u \cdot (\lambda \circ \mu)$
\item \label{axiomneut0} $u \cdot \pi_{\dom(u)} = u$
\item \label{axiomdxx0} $d_{xx} \neq 0$
\item \label{axiomdxy0} $d_{xy} = d_{xx} \cdot \tfrac{xx}{xy}$
\item \label{axiompreimg0} $u \neq 0 \Rightarrow \dom(u \cdot \lambda) = \lambda^{-1}(\dom(u))$
\item \label{axiomfinite0} $u \neq 0 \Rightarrow \dom(u) \text{ is finite}$
\item \label{axiomdom0} $\dom(u) = \{u \in \mathbf{var} \mid u \leq d_{xx}\}$
\end{enumerate}
for all $u,v \in V$, $\lambda,\mu \in \mathcal{T}_{\mathrm{fp}}(\mathbf{var})$, $x,y \in \mathbf{var}$
and $Z \in \mathfrak{P}_{\mathrm{fin}}(\mathbf{var})$.
\end{theorem}
\begin{proof}
Axioms~\ref{axiompreimg0}, \ref{axiomsemi0} and~\ref{axiomneut0} have already been discussed at the beginning of Sect.~\ref{rep0}.
Axiom~\ref{axiomzero0} holds by definition of "$\cdot$". Axioms~\ref{axiomdxx0}, \ref{axiomfinite0} and~\ref{axiomdom0} are projectional
semilattice axioms.
\ref{axiomone0}: We have $u \cdot \pi_{\emptyset} \in \mathbf{V}^{*}[\emptyset]$ by~\ref{axiompreimg0},
so $u \cdot \pi_{\emptyset} = 1$ by~\ref{psaxone0}.
\ref{axiomdist0}: For $u=0$ or $v=0$ by~\ref{axiomzero0}, so let $u \in \mathbf{V}^{*}[X]$ and $v \in \mathbf{V}^{*}[Y]$.
Then $(u \wedge v) \cdot \pi_Z = C_{(X \cup Y)\setminus Z}(u \wedge v)$ by~\ref{axiomzero0}
if $u \wedge v = 0$, and otherwise by Prop.~\ref{domains0}\ref{domains2} and Prop.~\ref{proj0}\ref{proj2} (where WLOG $Z \subseteq X \cup Y$).
By assumption $X \subseteq Z$, so $((X \cup Y)\setminus Z) \cap X = \emptyset$ in Prop.~\ref{gency0}\ref{gency3},
and thus $C_{(X \cup Y) \setminus Z}(u \wedge v) = u \wedge C_{(X \cup Y)\setminus Z}(v)$.
Lastly, $C_{(X \cup Y)\setminus Z}(v) = C_{Y\setminus Z}(C_{(X\setminus Y)\setminus Z}(v)) = C_{Y\setminus Z}(v) = v \cdot \pi_{Z}$
by Prop.~\ref{psaxdom0} (showing $C_{(X\setminus Y)\setminus Z}(v) = v$) and Prop.~\ref{proj0}\ref{proj2},
so altogether $(u \wedge v) \cdot \pi_Z = u \wedge (v \cdot \pi_Z)$.
\ref{axiomproj0}: We have $u \leq C_{X\setminus Z}(u) = u \cdot \pi_Z$ by Prop.~\ref{derived0}\ref{derived2} and Prop.~\ref{proj0}\ref{proj2}.
\ref{axiommult0}: For $u=0$ or $v=0$ by~\ref{axiomzero0}, so let $u \in \mathbf{V}^{*}[X]$ and $v \in \mathbf{V}^{*}[Y]$.
First, note that $u \cdot \lambda \leq u \cdot \lambda \cdot \pi_{\lambda^{-1}(Y)} = u \cdot \pi_{Y} \cdot \lambda = C_{X\setminus Y}(u) \cdot \lambda$
by~\ref{axiomproj0}, \ref{axiomsemi0} and Prop.~\ref{proj0}\ref{proj2}. Moreover,
$C_{X\setminus Y}(u) \leq C_{X\setminus Y}(v) = v$ by Prop.~\ref{derived0}\ref{derived2} and~\ref{psaxdom0}.
From $u \leq v$ follows $Y \subseteq X$, so $C_{X\setminus Y}(u) \in \mathbf{V}^{*}[Y]$ by Prop.~\ref{domains0}\ref{domains3},
and thus $C_{X\setminus Y}(u) \cdot \lambda = C_{X\setminus Y}(u) \odot \lambda\dbar^{Y} \leq v \odot \lambda\dbar^{Y} = v \cdot \lambda$

$C_{X\setminus Y}(u) \leq C_{X\setminus Y}(v) = v$, then 
$u \cdot \lambda \leq u \cdot \lambda \cdot \pi_{\lambda^{-1}(Y)} = u \cdot \pi_{Y} \cdot \lambda
= C_{X\setminus Y}(u) \cdot \lambda = C_{X\setminus Y}(u) \odot \lambda\dbar^Y \leq v \odot \lambda\dbar^Y = v \cdot \lambda$
\ref{axiomequal0}: $(u \cdot \pi_{X\setminus\{y\}}) \wedge d_{xy} = c_y(u) \wedge d_{xy} = u$
\ref{axiomdxy0}: $d_{xx} \cdot \tfrac{xx}{xy} = d_{xx} \odot \tfrac{xx}{xy}\dbar^{\{x\}} = d_{xx} \wedge (d_{xx} \wedge d_{xy}) = d_{xy}$
by Prop.~\ref{proj0}\ref{proj3} and Prop.~\ref{domains0}\ref{domains1}.
\qed
\end{proof}
\begin{theorem}
\label{iso0}
Every projectional semilattice is isomorphic to a conjunctive table algebra (with equality).
\end{theorem}
\begin{proof}
Let $(V,\wedge,0,1,c_x,d_{xy},\dom)_{x,y \in \mathbf{var}}$ be a projectional semilattice. By Thm.~\ref{orbital0}, the reduct
$(V,\wedge,0,1,\cdot,d_{xy},\dom)_{x,y \in \mathbf{var}}$ of the orbital extension is an orbital semilattice. As shown in~\cite{KS22a},
every orbital semilattice is isomorphic to an orbital table algebra
$(A,\Join,\emptyset,\{\langle\rangle\},\circ,E_{xy},\mathrm{schema})_{x,y \in \mathbf{var}}$, i.e. there is a bijection
$\varphi:V \rightarrow A$ which preserves all orbital semilattice operations and constants.
In addition, since $c_x u = u \cdot \pi_{\dom(u)\setminus\{x\}}$ by Prop.~\ref{proj0}, we obtain
$\varphi(c_x u) = \varphi(u \cdot \pi_{\dom(u)\setminus\{x\}}) = \varphi(u) \circ \pi_{\mathrm{schema}(\varphi(u))\setminus\{x\}} = \del_x(\varphi(u))$.
So $\varphi$ preserves all projectional semilattice operations and constants. This shows that $(V,\wedge,0,1,c_x,d_{xy},\dom)$ is isomorphic
to the conjunctive table algebra $(A,\Join,\emptyset,\{\langle\rangle\},\circ,E_{xy},\mathrm{schema})_{x,y \in \mathbf{var}}$.
\qed
\end{proof}



\bibliographystyle{splncs04}
\bibliography{cta}

\appendix
\section{Proofs for Sect.~\ref{pslat0}}
\label{proofs0}
\begin{proof}[Proof of Prop.~\ref{derived0}]
\ref{derived1} We obtain $c_x(c_x(v)) = c_x(c_x(v) \wedge c_x(v)) = c_x(c_x(v)) \wedge c_x(v) \leq c_x(v)$ using~\ref{psaxdist0},
and $c_x(v) \leq c_x(c_x(v))$ from~\ref{psaxproj0}.
\ref{derived2} If $u \leq v$, then $u \leq v \leq c_x(v)$ by~\ref{psaxproj0}, so $u = u \wedge c_x(v)$, and thus
$c_x(u) = c_x(u \wedge c_x(v)) = c_x(u) \wedge c_x(v) \leq c_x(v)$ by~\ref{psaxdist0}.
\ref{derived3}
If $x \not\in \dom(v)$, then $v \not\leq d_{xx}$ by~\ref{psaxdom0}, so $v \neq 0$, moreover $v=c_x(v)$ by~\ref{psaxdim0},
and thus $c_x(u \wedge v) = c_x(u) \wedge v$ by~\ref{psaxdist0}.
\ref{derived4}
Assume $x \neq y$ and $u \leq d_{xy}$. Then $u = d_{xy} \wedge u$, and thus $d_{xy} \wedge c_x(u) = d_{xy} \wedge c_x(d_{xy} \wedge u) \leq u$
by~\ref{psaxequal0}. Conversely, $u \leq c_x(u)$ by~\ref{psaxproj0}, so $u \leq d_{xy} \wedge c_x(u)$.
\ref{derived5}
If $x \neq y$, then $c_x(d_{xy}) = c_x(d_{yx} \wedge d_{xy}) = d_{yy}$ by~\ref{psaxsym0} and~\ref{psaxdxyz0}.
\ref{derived6}
If $z \neq x,y$, then $d_{xz} \wedge d_{zy} \leq c_z(d_{xz} \wedge d_{zy}) = d_{xy}$ by~\ref{psaxproj0} and~\ref{psaxdxyz0}.
Otherwise, if $z=x$, then $d_{xz} \wedge d_{zy} \leq d_{zy} = d_{xy}$, and likewise if $z=y$.
\qed
\end{proof}
\begin{proof}[Proof of Prop.~\ref{domains0}]
\ref{domains1}
First, we show $x,y \in \dom(d_{xy})$. For $x=y$, trivially $d_{xx} \leq d_{xx}$, so $x \in \dom(d_{xx})$ by~\ref{psaxdom0}.
For $x \neq y$, we obtain $d_{xy} \leq c_x(d_{xy}) = d_{yy}$ from~\ref{psaxproj0} and Prop.~\ref{derived0}\ref{derived5},
and symmetrically $d_{xy} = d_{yx} \leq d_{xx}$ using~\ref{psaxsym0}, which means $x,y \in \dom(d_{xy})$.
Next, we show $d_{xy} \neq 0$. For $x=y$ see~\ref{psaxdxx0}. For $x \neq y$, we have $c_x(d_{xy})\neq 0$ by Prop.~\ref{derived0}\ref{derived5}
and~\ref{psaxdxx0}, so $d_{xy} \neq 0$ by~\ref{psaxzero0}. Finally, we show $z \not\in \dom(d_{xy})$ for all $z \in \mathbf{var}\setminus\{x,y\}$.
We have $d_{xy} = c_z(d_{xz} \wedge d_{zy})$ by~\ref{psaxdxyz0}, and thus $c_z(d_{xy}) = d_{xy}$ by~\ref{derived0}\ref{derived1},
so $d_{xy} \neq d_{zz}$ by~\ref{psaxdom0}, i.e. $z \not\in \dom(d_{xy})$.
Taken together, this means $d_{xy} \in \mathbf{V}^{*}[\{x,y\}]$.

\ref{domains2}
If $u \wedge v = 0$, then $u \wedge v \in \mathbf{V}[X \cup Y]$ and we are done. Otherwise, $u \neq 0$ and $v \neq 0$.
Then for all $z \in \mathbf{var}\setminus(X \cup Y)$, we have $u \not\leq d_{zz}$, and thus $c_z(u \wedge v) \leq c_z(u) = u$
by Prop.~\ref{derived0}\ref{derived2} and~\ref{psaxdim0}. Likewise, we obtain $c_z(u \wedge v) \leq v$, so altogether
$c_z(u \wedge v) \leq u \wedge v$, which implies $c_z(u \wedge v) = u \wedge v$ by~\ref{psaxproj0},
and thus $u \wedge v \not\leq d_{zz}$ by~\ref{psaxdim0}. This shows $\dom(u \wedge v) \subseteq X \cup Y$. If $x \in X$,
then $u \wedge v \leq u \leq d_{xx}$, so $X \subseteq \dom(u \wedge v)$, and likewise $Y \subseteq \dom(u \wedge v)$.
Hence $u \wedge v \in \mathbf{V}[X \cup Y]$.

\ref{domains3}
It suffices to show $c_z(u) \in \mathbf{V}^{*}[Y\setminus\{z\}]$. Since $0 \neq u \leq c_z(u)$ by~\ref{psaxproj0}, we obtain $c_z(u) \neq 0$,
and $\dom(c_z(u)) \subseteq \dom(u) = Y$ from~\ref{psaxdom0}. Moreover, $c_z(c_z(u))=c_z(u)$ by Prop.~\ref{derived0}\ref{derived1},
so $c_z(u) \not\leq d_{zz}$ by~\ref{psaxdim0}, i.e. $\dom(c_z(u)) \subseteq Y\setminus\{z\}$. Let $y \in Y\setminus\{z\}$. Since
$u \leq d_{yy}$, we obtain $c_z(u) \leq c_z(d_{yy}) = d_{yy}$ using Prop.~\ref{derived0}\ref{derived2},
$d_{yy} \not\leq d_{zz}$ by~\ref{domains1}, and~\ref{psaxdim0}. This shows $Y\setminus\{z\} \subseteq \dom(c_z(u))$.
So altogether $c_z(u) \in \mathbf{V}^{*}[Y\setminus\{z\}]$.

\ref{domains4}
Since $d_{xy} \in \mathbf{V}^{*}[\{x,y\}]$ by~\ref{domains1}, we obtain
$e_{\varrho} \in \mathbf{V}[\bigcup_{(x,y) \in \varrho} \{x,y\}]$ from~\ref{domains2},
i.e. $e_{\varrho} \in \mathbf{V}[\mathrm{field}(\varrho)]$. It remains to show $e_{\varrho} \neq 0$.
First, consider the relations $\theta_n:=\{(x_k,x_1) \mid 1 \leq k \leq n\}$, defined on the first $n$ variables of $\mathbf{var}$.
We have $e_{\theta_0} = e_{\emptyset} = 1 \neq 0$, and $e_{\theta_1} = d_{x_1x_1} \neq 0$ by~\ref{psaxdxx0}.
Now assume $e_{\theta_n} \neq 0$ for $n \geq 1$. Then $\dom(e_{\theta_n})=\{x_1,\dots,x_n\}$, and thus
$c_{x_{n+1}}(e_{\theta_{n+1}}) = c_{x_{n+1}}(d_{x_{n+1}x_1} \wedge e_{\theta_n}) = c_{x_{n+1}}(d_{x_{n+1}x_1}) \wedge e_{\theta_n}
= d_{x_1x_1} \wedge e_{\theta_n} = e_{\theta_n}$ by Prop.~\ref{derived0}\ref{derived3}, Prop.~\ref{derived0}\ref{derived5} and~\ref{psaxdom0},
in particular $c_{x_{n+1}}(e_{\theta_{n+1}}) \neq 0$, which implies $e_{\theta_{n+1}} \neq 0$ by~\ref{psaxzero0}.
So by induction, $e_{\theta_n} \neq 0$ for all $n \in \mathbb{N}$.
Now consider any $\varrho \subseteq X \times Y$ with $X,Y \in \mathfrak{P}_{\mathrm{fin}}(\mathbf{var})$.
For sufficiently large $n$, we have $X \cup Y \subseteq \{x_1,\dots,x_n\}$,
so for all $(x,y) \in \varrho$, we obtain $e_{\theta_n} \leq d_{xx_1} \wedge d_{yx_1} = d_{xx_1} \wedge d_{x_1y} \leq d_{xy}$
using~\ref{psaxsym0} and Prop.~\ref{derived0}\ref{derived5}, which shows $e_{\theta_n} \leq e_{\varrho}$. Hence $e_{\varrho} \neq 0$.
\qed
\end{proof}
\begin{proof}[Proof of Prop.~\ref{eqcomp0}]
\ref{eqcomp1}
For all $x \in X$, we have $e_{\mu} \wedge e_{\nu} \leq d_{x\mu(x)} \wedge d_{\mu(x)\nu(\mu(x))} \leq d_{x\mu(x)}$ by Prop.~\ref{derived0}\ref{derived6}.
So $e_{\mu} \wedge e_{\nu} \leq e_{\nu \circ \mu}$.
\ref{eqcomp2}
For all $x \in X$, we have $e_{\nu\circ\mu} \wedge e_{\nu} \leq d_{x\nu(\mu(x))} \wedge d_{\mu(x)\nu(\mu(x))} \leq d_{x\mu(x)}$ by~\ref{psaxsym0}
and Prop.~\ref{derived0}\ref{derived6}. So $e_{\nu\circ\mu} \wedge e_{\nu} \leq e_{\mu}$. Together with~\ref{eqcomp1}, we obtain~\ref{eqcomp2}.
\ref{eqcomp3}
Since $\mu$ is a folding, we have $Z \subseteq X$, and $\mu(\mu(x))=\mu(x)$ for all $x \in X$.
So for all $x \in X$, we have $e_{\nu\circ\mu} \leq d_{x\nu(\mu(x))} \wedge d_{\mu(x)\nu(\mu(\mu(x)))} \leq d_{x\mu(x)}$
by~\ref{psaxsym0}, the previous equation, and Prop.~\ref{derived0}\ref{derived6}. This shows $e_{\nu\circ\mu} \leq e_{\mu}$.
Also, $\mu(z)=z$ for all $z \in Z$. So $e_{\nu\circ\mu} \leq d_{z\nu(\mu(z))} = d_{z\nu(z)}$, which shows $e_{\nu\circ\mu} \leq e_{\nu}$.
Taken together, we have $e_{\nu\circ\mu} \leq e_{\mu} \wedge e_{\nu}$. Combined with~\ref{eqcomp1}, this shows~\ref{eqcomp3}.
\qed
\end{proof}
\begin{proof}[Proof of Prop.~\ref{gencydiag0}]
Both statements are shown by induction over the cardinality $\#X$.

\ref{gencydiag1}
For $X=\emptyset$, we have $C_{\emptyset}(u \wedge e_{\emptyset}) = u \wedge 1 = u$. Assume~\ref{gencydiag1} holds for a given $X$.
Let $X':=X \cup \{x\}$ for some $x \not\in X$, and let $\lambda:X' \rightarrow Y$ be domain-disjoint. Then
\begin{align*}
C_{X'}&(u \wedge e_{\lambda}) \underset{\text{Defs.}}{=} c_xC_X(u \wedge e_{\lambda|_X} \wedge d_{x\lambda(x)})
\underset{\text{Prop.\,\ref{gency0}\ref{gency3}}}{=} c_x(C_X(u \wedge e_{\lambda|_X}) \wedge d_{x\lambda(x)}) \\
&\underset{\text{I.H.}}{=} c_x(u \wedge d_{x\lambda(x)})
\underset{\text{Prop.}\,\ref{derived0}\ref{derived3}}{=} u \wedge c_x(d_{x\lambda(x)})
\underset{\text{Prop.}\,\ref{derived0}\ref{derived5}}{=} u \wedge d_{\lambda(x)\lambda(x)}
\underset{\ref{psaxdom0}}{=} u \quad,
\end{align*}
where the conditions $X \cap \{x,\lambda(x)\} = \emptyset$, $x \not\in Y$, $x \neq \lambda(x)$ and $\lambda(x) \in Y$,
which must be checked for the second, fourth, fifth and last equality, respectively, can be verified.

\ref{gencydiag2}
For $X=\emptyset$, we have $C_{\emptyset}(u) \wedge e_{\emptyset} = u \wedge 1 = u$. Assume~\ref{gencydiag2} holds for a given $X$.
Let $X':=X \cup \{x\}$ for some $x \not\in X$, and let $\lambda:X' \rightarrow Y$ be domain-disjoint. Then for all $u \leq e_{\lambda}$, we obtain
\begin{align*}
C_{X'}(u) \wedge e_{\lambda} \underset{\text{Defs.}}{=} C_X(c_x(u)) \wedge d_{x\lambda(x)} \wedge e_{\lambda|_X}
&\underset{\text{Prop.}\,\ref{gency0}\ref{gency3}}{=} C_X(c_x(u) \wedge d_{x\lambda(x)}) \wedge e_{\lambda|_X} \\
&\underset{\text{Prop.}\,\ref{derived0}\ref{derived4}}{=} C_X(u) \wedge e_{\lambda|_X}
\underset{\text{I.H.}}{=} u \quad,
\end{align*}
where the conditions $X \cap \{x,\lambda(x)\} = \emptyset$, $x \neq \lambda(x)$, $u \leq d_{x\lambda(x)}$, $u \leq e_{\lambda|_X}$,
which must be checked for the second, third, third and last equality, respectively, can be verified.
\qed
\end{proof}
\begin{proof}[Proof of Prop.~\ref{ocomp0}]
Let $(\circledast_{XY})_{X,Y \in \mathfrak{P}_{\mathrm{fin}}(\mathbf{var})}$ and
$(\circledcirc_{XY})_{X,Y \in \mathfrak{P}_{\mathrm{fin}}(\mathbf{var})}$ be families of operations
satisfying~\ref{pseaxdis0} and~\ref{pseaxsemi0}. Then
\begin{align*}
u \circledast_{XY} \lambda &\underset{\ref{pseaxsemi0}}{=} (u \circledast_{ZY} \xi_{XY}) \circledast_{XZ} (\xi_{XY}^{-1} \circ \lambda) \\
&\underset{\ref{pseaxdis0}}{=} (u \circledcirc_{ZY} \xi_{XY}) \circledcirc_{XZ} (\xi_{XY}^{-1} \circ \lambda)
\underset{\ref{pseaxsemi0}}{=} u \circledcirc_{XY} \lambda
\end{align*}
for all $u \in \mathbf{V}^{*}[Y]$ and $\lambda:X \rightarrow Y$, which shows uniqueness.

For domain-disjoint $\lambda$, we obtain $C_X(u \wedge e_{\lambda}) = u \neq 0$ using Prop.~\ref{gencydiag0}\ref{gencydiag1},
which implies $u \wedge e_{\lambda} \neq 0$ by Prop.~\ref{gency0}\ref{gency1};
so $u \wedge e_{\lambda} \in V^{*}[X \cup Y]$ by Props.~\ref{domains0}\ref{domains4}
and~\ref{domains0}\ref{domains2}, and thus $C_Y(u \wedge e_{\lambda}) \in V^{*}[X]$ by Prop.~\ref{domains0}\ref{domains3}.
So we obtain $u \odot \lambda \in V^{*}[X]$ from~\eqref{outercomp0}, also for non-domain-disjoint $\lambda$,
which shows $\odot_{XY}:V^{*}[Y] \times Y^X \rightarrow V^{*}[X]$.

Axiom~\ref{pseaxdis0} is satisfied by the definition in~\eqref{outercomp0}. For showing~\ref{pseaxsemi0}, we need to consider different cases.

\textit{Case 1: $\nu$ and $\mu$ domain-disjoint.}
\textit{Subcase 1.1: $\nu \circ \mu$ domain-disjoint.}
By Prop.~\ref{domains0}\ref{domains4} we have $\dom(e_{\nu \circ \mu}) \subseteq X \cup Y$ and $\dom(e_{\mu}) \subseteq X \cup Z$.
Since $X$, $Y$ and $Z$ are pairwise disjoint, this implies $Z \cap \dom(e_{\nu \circ \mu}) = \emptyset$ and $Y \cap \dom(e_{\mu}) = \emptyset$,
which is used in the two applications of Prop.~\ref{gency0}\ref{gency3} below; we obtain
\begin{align*}
u \odot (\nu \circ \mu) &\underset{\eqref{outercomp0}}{=} C_Y(u \wedge e_{\nu\circ\mu})
\underset{\text{Prop.\,\ref{gencydiag0}\ref{gencydiag1}}}{=} C_Y(C_Z(u \wedge e_{\nu}) \wedge e_{\nu\circ\mu}) \\
&\underset{\text{Prop.\,\ref{gency0}\ref{gency3}}}{=} C_Y(C_Z(u \wedge e_{\nu} \wedge e_{\nu\circ\mu}))
\underset{\text{Prop.\,\ref{eqcomp0}\ref{eqcomp2}}}{=} C_Y(C_Z(u \wedge e_{\nu} \wedge e_{\mu})) \\
&\underset{\text{Prop.\,\ref{gency0}\ref{gency2}}}{=} C_Z(C_Y(u \wedge e_{\nu} \wedge e_{\mu}))
\underset{\text{Prop.\,\ref{gency0}\ref{gency3}}}{=} C_Z(C_Y(u \wedge e_{\nu}) \wedge e_{\mu}) \\
&\underset{\eqref{outercomp0}}{=} u \odot \nu \odot \mu \quad.
\end{align*}

\textit{Subcase 1.2: $\nu \circ \mu$ not domain-disjoint.}
Set $\lambda:=\nu \circ \mu$. For each $i \in \{1,2,3\}$, we define domain-disjoint functions $\mu_i:X \rightarrow Z_i$ and $\nu_i:Z_i \rightarrow Y$
as follows: for $i=1$, we set $\nu_i:=\xi_{XY}$ and $\mu_i:=\xi_{XY}^{-1} \circ \lambda$; for $i=3$, we simply set $\nu_i:=\nu$ and $\mu_i:=\mu$;
for $i=2$, let $Z_2 \in \mathfrak{P}_{\mathrm{fin}}(\mathbf{var})$ with $Z_2 \cap (X \cup Y \cup Z_1 \cup Z_3) = \emptyset$ and $\#Z_2=\#Y$;
by the latter, there exists a bijection $\nu_i:Z_2 \rightarrow Y$, and we set $\mu_i:=\nu_i^{-1} \circ \lambda$.
Since $\lambda$ is not domain-disjoint, we have $u \odot \lambda = u \odot \nu_1 \odot \mu_1$ by~\eqref{outercomp0},
and next we show $u \odot \nu_i \odot \mu_i = u \odot \nu_{i+1} \odot \mu_{i+1}$ for $i \in \{1,2\}$.

Note that $\nu_1$ and $\nu_2$ are bijective, and $\nu_i \circ \mu_i = \nu_{i+1} \circ \mu_{i+1}$ for $i \in \{1,2\}$,
i.e. $\mu_i = \nu_i^{-1} \circ \nu_{i+1} \circ \mu_{i+1}$, so we obtain
\begin{align*}
u \odot \nu_i \odot \mu_i = u \odot \nu_i \odot (\nu_i^{-1} \circ \nu_{i+1} \circ \mu_{i+1})
&\underset{\text{Subcase 1.1}}{=} u \odot \nu_i \odot (\nu_i^{-1} \circ \nu_{i+1}) \odot \mu_{i+1} \\
&\underset{\text{Subcase 1.1}}{=} u \odot \nu_{i+1} \odot \mu_{i+1} \quad,
\end{align*}
where Subcase 1.1 applies in the second equation since $Z_i$,$Z_{i+1}$ and $X$ are pairwise disjoint,
and in the last equation since $Z_i$, $Z_{i+1}$ and $Y$ are pairwise disjoint.
Altogether, we have $u \odot (\nu \circ \mu) = u \odot \lambda = u \odot \nu_1 \odot \mu_1 = \dots = u \odot \nu_3 \odot \mu_3 = u \odot \nu \odot \mu$.

\textit{Case 2: general case.}
As before, we decompose $\mu:X \rightarrow Z$ into domain-disjoint $\tau_1:X \rightarrow Z_1$ and $\sigma_1:Z_1 \rightarrow Z$;
and $\nu:Z \rightarrow Y$ into domain-disjoint $\tau_2:Z \rightarrow Z_2$ and $\sigma_2:Z_2 \rightarrow Y$,
where $Z_1$ and $Z_2$ are chosen disjoint from $X,Y,Z$ and each other. Noting that $\tau_2 \circ \sigma_1 \circ \tau_1:X \rightarrow Z_2$
and $\tau_2 \circ \sigma_1:Z_1 \rightarrow Z_2$ are then domain-disjoint, as well, we obtain
\begin{align*}
u \odot (\nu \circ \mu) &= u \odot (\sigma_2 \circ \tau_2 \circ \sigma_1 \circ \tau_1)
\underset{\text{Case 1}}{=} u \odot \sigma_2 \odot (\tau_2 \circ \sigma_1 \circ \tau_1) \\
&\underset{\text{Case 1}}{=} u \odot \sigma_2 \odot (\tau_2 \circ \sigma_1) \odot \tau_1
\underset{\text{Case 1}}{=} u \odot \sigma_2 \odot \tau_2 \odot \sigma_1 \odot \tau_1 \\
&\underset{\text{Case 1}}{=} u \odot (\sigma_2 \circ \tau_2) \odot (\sigma_1 \circ \tau_1)
= u \odot \nu \odot \mu \quad,
\end{align*}
which concludes the proof.
\qed
\end{proof}
\begin{proof}[Proof of Prop.~\ref{proj0}]
\ref{proj1}
By~\ref{psaxsym0} we have $e_{\sigma} = e_{\sigma^{-1}}$, and thus $u \wedge e_{\sigma} \leq e_{\sigma^{-1}}$,
which allows using Prop.~\ref{gencydiag0}\ref{gencydiag2} below; we have
\begin{align*}
u \odot \sigma \odot \sigma^{-1} &\underset{\ref{pseaxdis0}}{=} C_X(C_Y(u \wedge e_{\sigma}) \wedge e_{\sigma^{-1}}) \\
&\underset{\text{Prop.\,\ref{gencydiag0}\ref{gencydiag2}}}{=} C_X(u \wedge e_{\sigma^{-1}})
\underset{\ref{psaxsym0}}{=} C_X(u \wedge e_{\sigma})
\underset{\text{Prop.\,\ref{gencydiag0}\ref{gencydiag1}}}{=} u \quad.
\end{align*}
\ref{proj2}
Let $\iota_X:X \rightarrow Y$ be an inclusion. This implies $X \subseteq Y$. Let $\sigma:Z \rightarrow X$
be a bijection where $Z \cap Y \neq \emptyset$. Then both $\sigma$ and $\iota_X \circ \sigma$ are domain-disjoint.
Moreover, $\dom(e_{\sigma}) = Z \cup X$ by Prop.~\ref{domains0}\ref{domains4},
and thus $(Y \setminus X) \cap \dom(e_{\sigma}) = \emptyset$, which allows using Prop.~\ref{gency0}\ref{gency3} below.
Trivially, $e_{\iota_X \circ \sigma} = e_{\sigma}$, which is used in the third equation below;
and Prop.~\ref{domains0}\ref{domains3} provides $C_{Y\setminus X}(u) \in V^{*}[X]$, which is required for the last equation below.
All things considered, we obtain
\begin{align*}
u \odot \iota_X \odot \sigma &\underset{\ref{pseaxsemi0}}{=} u \odot (\iota_X \circ \sigma)
\underset{\ref{pseaxdis0}}{=} C_Y(u \wedge e_{\iota_X \circ \sigma})
= C_XC_{Y\setminus X}(u \wedge e_{\sigma}) \\
&\underset{\text{Prop.\,\ref{gency0}\ref{gency3}}}{=} C_X(C_{Y\setminus X}(u) \wedge e_{\sigma})
\underset{\ref{pseaxdis0}}{=} C_{Y\setminus X}(u) \odot \sigma \quad.
\end{align*}
As~\ref{proj1} states, $\sigma^{-1}$ cancels $\sigma$, so $u \odot \iota_X = C_{Y\setminus X}(u)$.

\ref{proj3}
Let $\delta:X \rightarrow Y$ be a folding. This implies $Y \subseteq X$. Let $\sigma:Z \rightarrow Y$ be a bijection where $Z \cap X = \emptyset$.
Then both $\sigma^{-1}$ and $\sigma^{-1} \circ \delta$ are domain-disjoint. We abbreviate $v:=u \odot \sigma$, and note that $v \in \mathbf{V}^{*}[Z]$.
We obtain
\begin{align*}
v \odot \sigma^{-1} \odot \delta \underset{\ref{pseaxsemi0}}{=} v \odot (\sigma^{-1} \circ \delta)
&\underset{\ref{pseaxdis0}}{=} C_Z(v \wedge e_{\sigma^{-1} \circ \delta})
\underset{\text{Prop.\,\ref{eqcomp0}\ref{eqcomp3}}}{=} C_Z(v \wedge e_{\sigma^{-1}} \wedge e_{\delta}) \\
&\underset{\text{Prop.\,\ref{gency0}\ref{gency3}}}{=} C_Z(v \wedge e_{\sigma^{-1}}) \wedge e_{\delta}
\underset{\ref{pseaxdis0}}{=} (v \odot \sigma^{-1}) \wedge e_{\delta} \quad,
\end{align*}
where $Z \cap \dom(e_{\delta}) = Z \cap X = \emptyset$ allows using Prop.~\ref{gency0}\ref{gency3}. Resolving $v$,
we obtain $u \odot \delta = u \wedge e_{\delta}$ from~\ref{proj1}.
\qed
\end{proof}

\end{document}